\documentclass[11pt]{article}
\usepackage{amsmath}
\usepackage{amssymb}
\usepackage{amsthm}
\usepackage[usenames]{color}
\usepackage{amscd}
\usepackage{indentfirst}

\usepackage[colorlinks=true,linkcolor=blue,filecolor=red,
citecolor=webgreen]{hyperref}
\definecolor{webgreen}{rgb}{0,.5,0}

\def\N{{\Bbb N}}

\def\R{{\Bbb R}}

\def\1{{\bf 1}}

\def\id{\operatorname{id}}
\def\lcm{\operatorname{lcm}}

\newtheorem{theorem}{Theorem}[section]
\newtheorem{lemma}[theorem]{Lemma}
\newtheorem{corollary}{Corollary}

\newtheorem{remark}{Remark}

\begin{document}

\title{{\bf On the average value of the least common multiple of $k$ positive
integers}}

\author{Titus Hilberdink and L\'aszl\'o T\'oth}
\date{}
\maketitle

\centerline{Journal of Number Theory {\bf 169} (2016), 327--341}

\abstract{We deduce an asymptotic formula with error term for
the sum $\sum_{n_1,\ldots,n_k \le x} f([n_1,\ldots, n_k])$, where
$[n_1,\ldots, n_k]$ stands for the least common multiple of the
positive integers $n_1,\ldots, n_k$ ($k\ge 2$) and $f$ belongs to a
large class of multiplicative arithmetic functions, including, among
others, the functions $f(n)=n^r$, $\varphi(n)^r$, $\sigma(n)^r$ ($r>-1$
real), where $\varphi$ is Euler's totient function and $\sigma$ is
the sum-of-divisors function. The proof is by elementary arguments,
using the extension of the convolution method for arithmetic
functions of several variables, starting with the observation that
given a multiplicative function $f$, the function of $k$ variables
$f([n_1,\ldots,n_k])$ is multiplicative.}

\medskip



\section{Introduction}

We use the following notation: $\N =\{1,2,\ldots\}$, $*$ is the Dirichlet
convolution of arithmetic functions,
$\id_r$ ($r\in \R$) is the function $\id_r(n)=n^r$ ($n\in \N$), $\1=\id_0$, $\id=\id_1$,
$\mu$ denotes the M\"obius function, $\lambda$ is the Liouville function, $\sigma_r=\1 *\id_r$,
$\sigma=\sigma_1$
is the sum-of-divisors function, $\tau=\sigma_0$ is the divisor function, $\beta_r=\lambda*\id_r$,
$\beta=\beta_1$ is
the alternating sum-of-divisors function (cf. \cite{Tot2013}), $\varphi_r=\mu * \id_r$ is
the generalized Euler function,
$\varphi=\varphi_1$ is Euler's totient function,  $\psi_r=\mu^2 * \id_r$ is the
generalized Dedekind function,
$\psi=\psi_1$ is the classical Dedekind function. If $n\in \N$, then $n=\prod_p p^{\nu_p(n)}$ is
its prime power factorization,
the product being over the primes $p$, where all but a
finite number of the exponents $\nu_p(n)$ are zero.

Furthermore, let $(n_1,\ldots,n_k)$ and $[n_1,\ldots,n_k]$ denote the greatest common
divisor (gcd) and the least common multiple (lcm) of $n_1,\ldots,n_k\in \N$ ($k\ge 2$), respectively.

It is easy to see that for any arithmetic function $f$ we have the
identity
\begin{equation} \label{sum_f_gcd}
\sum_{n_1,\ldots,n_k\le x} f((n_1,\ldots,n_k)) = \sum_{d\le x}
(\mu*f)(d) \left\lfloor \frac{x}{d} \right\rfloor^k,
\end{equation}
which leads to asymptotic formulas for this sum. For example, if
$f=\id$ and $k\ge 3$, then we have
\begin{equation} \label{gcdk}
\sum_{n_1,\ldots,n_k\le x} (n_1,\ldots,n_k) =
\frac{\zeta(k-1)}{\zeta(k)} x^k  + O(R_k(x)),
\end{equation}
where $R_3(x)=x^2\log x$ and $R_k(x)=x^{k-1}$ for $k\ge 4$. The case
$f=\id$, $k=2$ can be treated separately by writing
\begin{equation*}
\sum_{m,n\le x} (m,n) = 2\sum_{m\le n \le x} (m,n) - \sum_{n\le x} n
\end{equation*}
\begin{equation*}
 = 2\sum_{n\le x} (\mu* \id \tau)(n)
- \frac{x^2}{2}+O(x),
\end{equation*}
giving, by using elementary arguments, the formula
\begin{equation} \label{gcd2}
\sum_{m,n\le x} (m,n)= \frac{x^2}{\zeta(2)}\left(\log x+ 2\gamma
-\frac1{2}-\frac{\zeta(2)}{2}- \frac{\zeta'(2)}{\zeta(2)} \right) +
O(x^{1+\theta+\varepsilon}),
\end{equation}
valid for every $\varepsilon>0$, where $\gamma$ is Euler's constant
and $\theta$ is the exponent appearing in Dirichlet's divisor
problem.

For the lcm of $k$ positive integers there is no formula similar to \eqref{sum_f_gcd}.
However, in the case $k=2$, the lcm
of the integers $m,n\in \N$ can be written using their gcd as $[m,n]=mn/(m,n)$,
which enables to establish the following
asymptotic formula, valid for any positive real number $r$:
\begin{equation} \label{lcm2r}
\sum_{m,n\le x} [m,n]^r = \frac{\zeta(r+2)}{\zeta(2)} \cdot
\frac{x^{2(r+1)}}{(r+1)^2} + O(x^{2r+1}\log x).
\end{equation}

If $r\in \N$, then the error term in \eqref{lcm2r} can be improved
into $O(x^{2r+1}(\log x)^{2/3}(\log \log x)^{4/3})$, which is a
consequence of the result of Walfisz \cite[Satz 1, p.\ 144]{Wal1963}
for $\sum_{n\le x} \varphi(n)$.

For $k=2$ the asymptotic formulas concerning $\sum_{m,n\le
x} (m,n)^r$ and \\ $\sum_{m,n\le x} [m,n]^r$ are equivalent to those
for $\sum_{n\le x} g_r(n)$ and $\sum_{n\le x} \ell_r(n)$,
respectively, where $g_r(n)=\sum_{1\le j\le n} (j,n)^r$ is the
gcd-sum function and $\ell_r(n)=\sum_{1\le j\le n} [j,n]^r$ is the
lcm-sum function. The function $g_1(n)=\sum_{1\le j\le n} (j,n)$, investigated
by S.~S.~Pillai \cite{Pil1933}, is also called Pillai's function in the literature.

The above and related results go back, in chronological order, to
the work of E.~Ces\`{a}ro \cite{Ces1885}, E.~Cohen
\cite{Coh1960I,Coh1962II,Coh1961III}, K.~Alladi \cite{All1975},
P.~Diaconis and P.~Erd\H os \cite{DiaErd1977}, J.~Chidambaraswamy
and R.~Sitaramachandrarao \cite{ChiSit1985}, K.~A.~Broughan
\cite{Bro2001}, O.~Bordell\`{e}s \cite{Bor2007a,Bor2007b,Bor2010},
Y.~Tanigawa and W.~Zhai \cite{TanZha2008}, S.~Ikeda and K.~Matsuoka
\cite{IkeMat2014}, and others.

For example, formula \eqref{gcd2} with the weaker error
$O(x^{3/2}\log x)$ was given in \cite[Th.\ 2, Eq.\
(1.4)]{DiaErd1977} and was recovered in \cite[Th.\ 4.7]{Bro2001}.
Formula \eqref{gcd2} with the above error term  was established in
\cite[Th.\ 3.1]{ChiSit1985} and recovered in \cite[Th.\
1.1]{Bor2007a} (in both papers for Pillai's function). Formula
\eqref{lcm2r} was established in \cite[Th.\ 2, Eq.\
(1.6)]{DiaErd1977}. The better error term for \eqref{lcm2r} in the case
$r\in \N$ was obtained in \cite [Th.\ 2]{IkeMat2014}. Asymptotic
formulas for \eqref{sum_f_gcd} in the case $k=2$ and for various
choices of the function $f$, including $f=\sigma$ and $f=\varphi$
were deduced in \cite{Bor2010,Coh1960I,Coh1962II,Coh1961III}.
See also the survey paper \cite{Tot2010}.

The result
\begin{equation*}
\sum_{m,n,q\le x} [m,n,q]^r \sim c_r \frac{x^{3(r+1)}}{(r+1)^3} \quad (x\to
\infty),
\end{equation*}
valid for $r\in \N$, without any error term and with a
computable constant $c_r$ given in an implicit form, was obtained by
J.~L.~Fern\'andez and P.~Fern\'andez \cite[Th.\ 3(b)]{FF2013NT}.
Their proof is by an ingenious method based on the identity
$[m,n,q](m,n)(m,q)(n,q)=mnq(m,n,q)$ ($m,n,q\in \N$) and using the
dominated convergence theorem. As far as we know, there are no other
asymptotic results in the literature for the sum
\begin{equation} \label{lcm_k}
\sum_{n_1,\ldots,n_k \le x} f([n_1,\ldots, n_k]),
\end{equation}
in the case $k\ge 3$, where $f$ is an arithmetic function. It seems that the method of \cite{FF2013NT}
can not be extended for $k\ge 3$, even in the case $f=\id_r$. Also, it is not possible to reduce
the estimation of the sum \eqref{lcm_k} to sums of a single
variable, like in \eqref{sum_f_gcd}.

In this paper we deduce an asymptotic formula with remainder term
for the sum \eqref{lcm_k}, where $k\ge 2$ and $f$ belongs to a large
class of multiplicative arithmetic functions, including the
functions $\id_r$ with $r>-1$ real and $\sigma_r$, $\beta_r$,
$\varphi_r$, $\psi_r$ with $r\ge 1/2$ real. The proof is by
elementary arguments, using the extension of the convolution method
for arithmetic functions of several variables starting with the
observation that given a multiplicative function $f$, the function
of $k$ variables $f([n_1,\ldots,n_k])$ is multiplicative and the
associated multiple Dirichlet series factorizes as an Euler product.
The same method was used by the second author \cite{Tot2016} for a
different problem. See the survey paper \cite{Tot2014} of the second
author for basic properties of multiplicative functions of several
variables and related convolutions.

We also extend to the $k$ dimensional case the formula
\begin{equation} \label{form_frac_2}
\sum_{m,n\le x} \frac{[m,n]}{(m,n)} = \frac{\pi^2}{60} x^4 +
O(x^3\log x),
\end{equation}
which can be obtained in a similar manner to the results
\eqref{gcdk} and \eqref{lcm2r}. Properties of the operation $m\circ
n= [m,n]/(m,n)$  were investigated by the first author
\cite{Hil2014}.

Note that the following recent result of different type, concerning the lcm
of several positive integers, was obtained by J.~Cilleruelo,
J.~Ru\'e, P.~\u{S}arka and A.~Zumalac\'arregui \cite{Cil2014}: $\lcm
\{a: a \in A \} = 2^{n(1+o(1))}$ for almost all subsets $A \subset
\{1, \ldots, n\}$.

\section{Main results}

Let $r\in \R$ be a fixed number. Let ${\cal A}_r$ denote the
class of complex valued multiplicative arithmetic functions
satisfying the following properties: there exist real constants
$C_1,C_2$ such that

\begin{equation} \tag{i} \label{i}
|f(p)-p^r|\le C_1p^{r-1/2} \ \text{ for every prime $p$},
\end{equation}
and
\begin{equation} \tag{ii} \label{ii}
|f(p^{\nu})|\le C_2 p^{\nu r} \ \text{ for every prime power $p^\nu$ with $\nu \ge 2$}.
\end{equation}

Note that conditions (i) and (ii) imply that
\begin{equation} \tag{iii} \label{iii}
|f(p^{\nu})|\le  C_3p^{\nu r} \ \text{ for every prime power $p^\nu$ with $\nu \ge 1$},
\end{equation}
where $C_3=\max(C_1+1, C_2)$.

Observe that $\id_r \in {\cal A}_r$ for every $r\in \R$, while
$\sigma_r$, $\beta_r$, $\varphi_r$, $\psi_r \in {\cal A}_r$ for
every $r\in \R$ with $r\ge 1/2$. The functions $f(n)=\sigma(n)^r,
\beta(n)^r, \varphi(n)^r, \psi(n)^r$ also belong to the class ${\cal
A}_r$ for every $r\in \R$. As other examples of functions in the
class ${\cal A}_r$, with $r\in \R$, we mention $\varphi^*(n)^r$,
$\sigma^*(n)^r$ and $\sigma^{(e)}(n)^r$, where $\varphi^*(n)=
\prod_{p\mid n} \left(p^{\nu_p(n)}-1 \right)$ is the unitary Euler
totient, $\sigma^*(n)=\prod_{p\mid n} \left(p^{\nu_p(n)}+1 \right)$
is the sum-of-unitary-divisors function and
$\sigma^{(e)}(n)=\prod_{p\mid n} \sum_{d\mid \nu_p(n)} p^d$ denotes
the sum of exponential divisors of $n$. Furthermore, if $f$ is a
bounded multiplicative function such that $f(p)=1$ for every prime
$p$, then $f\in {\cal A}_0$. In particular, $\mu^2 \in {\cal A}_0$.

We prove the following results.

\begin{theorem} \label{Th1} Let $k\ge 2$ be a fixed integer and let $f\in {\cal A}_r$ be a function, where $r>-1$ is
real. Then for every $\varepsilon >0$,
\begin{equation} \label{form}
\sum_{n_1,\ldots,n_k\le x} f([n_1,\ldots,n_k]) = C_{f,k}
\frac{x^{k(r+1)}}{(r+1)^k}
+O\left(x^{k(r+1)-\frac1{2}\min(r+1,1)+\varepsilon}\right),
\end{equation}
and
\begin{equation} \label{form_2}
\sum_{n_1,\ldots,n_k\le x} \frac{f([n_1,\ldots,n_k])}{(n_1\cdots
n_k)^r} = C_{f,k} x^k +O\left(x^{k-\frac1{2}\min(r+1,1)+ \varepsilon}\right),
\end{equation}
where
\begin{equation*}
C_{f,k}= \prod_p \left(1-\frac1{p}\right)^k \sum_{\nu_1,\ldots,
\nu_k=0}^{\infty}
\frac{f(p^{\max(\nu_1,\ldots,\nu_k)})}{p^{(r+1)(\nu_1 +\cdots
+\nu_k)}}.
\end{equation*}
\end{theorem}

Formula \eqref{form} shows that the average order of
$f([n_1,\ldots,n_k])$ is $C_{f,k}(n_1\cdots n_k)^r$, in the sense that
\begin{equation*}
\sum_{n_1,\ldots,n_k\le x} f([n_1,\ldots,n_k]) \sim
\sum_{n_1,\ldots,n_k\le x} C_{f,k} (n_1\cdots n_k)^r \quad (x\to
\infty).
\end{equation*}

From \eqref{form_2} we deduce that
\begin{equation*}
\lim_{x\to \infty} \frac1{x^k} \sum_{n_1,\ldots,n_k\le x} \frac{f([n_1,\ldots,n_k])}{(n_1\cdots
n_k)^r} = C_{f,k},
\end{equation*}
representing the mean value of the function
$f([n_1,\ldots,n_k])/(n_1\cdots n_k)^r$. See N.~Ushiroya \cite[Th.\
4]{Ush2012} and the second author  \cite[Prop.\ 19]{Tot2014} for
general results on mean values of multiplicative arithmetic
functions of several variables.

\begin{theorem} \label{Th2} Let $k\ge 2$ be a fixed integer and let $f\in {\cal A}_r$ be a function, where $r\ge 0$ is real. Then
for every $\varepsilon >0$,
\begin{equation} \label{form_3}
\sum_{n_1,\ldots,n_k\le x} f \left(\frac{[n_1,\ldots,n_k]}{(n_1,\ldots,
n_k)}\right) = D_{f,k} \frac{x^{k(r+1)}}{(r+1)^k}
+O\left(x^{k(r+1)-\frac1{2}+\varepsilon}\right),
\end{equation}
where
\begin{equation*}
D_{f,k}= \prod_p \left(1-\frac1{p}\right)^k \sum_{\nu_1,\ldots,
\nu_k=0}^{\infty}
\frac{f(p^{\max(\nu_1,\ldots,\nu_k)-\min(\nu_1,\ldots,\nu_k)})}{p^{(r+1)(\nu_1
+\cdots +\nu_k)}}.
\end{equation*}
\end{theorem}

In the case $f=\id_r$ we obtain from Theorem \ref{Th1} the next result:

\begin{corollary} \label{Cor1} Let $k\ge 3$ and $r>-1$ be a real number. Then
for every $\varepsilon >0$,
\begin{equation} \label{form_cor_1}
\sum_{n_1,\ldots,n_k\le x} [n_1,\ldots,n_k]^r = C_{r,k}
\frac{x^{k(r+1)}}{(r+1)^k}
+O\left(x^{k(r+1)-\frac1{2}\min(r+1,1)+\varepsilon}\right),
\end{equation}
and
\begin{equation*}
\sum_{n_1,\ldots,n_k\le x} \left(\frac{[n_1,\ldots,n_k]}{n_1\cdots
n_k}\right)^r = C_{r,k} x^k +O\left(x^{k-\frac1{2}\min(r+1,1)+ \varepsilon}\right),
\end{equation*}
where
\begin{equation*}
C_{r,k}= \prod_p \left(1-\frac1{p}\right)^k \sum_{\nu_1,\ldots,
\nu_k=0}^{\infty}
\frac{p^{r\max(\nu_1,\ldots,\nu_k)}}{p^{(r+1)(\nu_1 +\cdots
+\nu_k)}}.
\end{equation*}

In particular,
\begin{equation} \label{C3r}
C_{r,3} =\zeta(r+2)\zeta(2r+3) \prod_p \left(1-\frac{3}{p^2}
+\frac{2}{p^3}+\frac{2}{p^{r+2}}-\frac{3}{p^{r+3}}+\frac1{p^{r+5}}\right),
\end{equation}
\begin{equation*}
C_{r,4} = \zeta(r+2)\zeta(2r+3)\zeta(3r+4) \prod_p \left(1-\frac{6}{p^2}
+\frac{8}{p^3}-\frac{3}{p^4} + \frac{5}{p^{r+2}}-\frac{12}{p^{r+3}} +\frac{6}{p^{r+4}}+\frac{4}{p^{r+5}}
\right.
\end{equation*}
\begin{equation} \label{C4r}
\left. -\frac{3}{p^{r+6}} +
\frac{3}{p^{2r+3}}-\frac{4}{p^{2r+4}}-\frac{6}{p^{2r+5}}+
\frac{12}{p^{2r+6}}-
\frac{5}{p^{2r+7}}+\frac{3}{p^{3r+5}}-\frac{8}{p^{3r+6}}
 +\frac{6}{p^{3r+7}}-\frac1{p^{3r+9}} \right).
\end{equation}
\end{corollary}

In the case $f=\id_r$ we deduce from Theorem \ref{Th2}:

\begin{corollary} \label{Cor2} Let $k\ge 3$ and $r>0$ be a real number. Then
for every $\varepsilon >0$,
\begin{equation} \label{form_cor_2}
\sum_{n_1,\ldots,n_k\le x} \left(\frac{[n_1,\ldots,n_k]}{(n_1,\ldots,
n_k)}\right)^r = D_{r,k} \frac{x^{k(r+1)}}{(r+1)^k}
+O\left(x^{k(r+1)-\frac1{2}+\varepsilon}\right),
\end{equation}
where
\begin{equation*}
D_{r,k}= \prod_p \left(1-\frac1{p}\right)^k \sum_{\nu_1,\ldots,
\nu_k=0}^{\infty}
\frac{p^{r(\max(\nu_1,\ldots,\nu_k)-\min(\nu_1,\ldots,\nu_k))}}{p^{(r+1)(\nu_1
+\cdots +\nu_k)}}.
\end{equation*}

In particular,
\begin{equation*}
D_{r,3} =C_{r,3}\frac{\zeta(3r+3)}{\zeta(2r+3)}, \quad D_{r,4}
=C_{r,4}\frac{\zeta(4r+4)}{\zeta(3r+4)}.
\end{equation*}
\end{corollary}

We remark that in the case $k=2$ asymptotic formulas
\eqref{form_cor_1} and \eqref{form_cor_2} reduce to \eqref{lcm2r}
and \eqref{form_frac_2} (case $r=1$), respectively, but the latter ones have
better error terms. Note that $D_{r,2} =\zeta(2r+2)/\zeta(2)$.

Among other special cases we consider the functions $\sigma, \varphi
\in {\cal A}_1$ and $\mu^2\in {\cal A}_0$.

\begin{corollary} \label{Cor3} Let $k\ge 2$. Then
for every $\varepsilon >0$,
\begin{equation*}
\sum_{n_1,\ldots,n_k\le x} \sigma([n_1,\ldots,n_k]) = C_{\sigma,k}
\frac{x^{2k}}{2^k}
+O\left(x^{2k-1/2+\varepsilon}\right),
\end{equation*}
and
\begin{equation*}
\sum_{n_1,\ldots,n_k\le x} \frac{
\sigma([n_1,\ldots,n_k])}{n_1\cdots
n_k} = C_{\sigma,k} x^k +O\left(x^{k-1/2+ \varepsilon}\right),
\end{equation*}
where
\begin{equation*}
C_{\sigma,k}= \prod_p \left(1-\frac1{p}\right)^k \sum_{\nu_1,\ldots,
\nu_k=0}^{\infty}
\frac{\sigma(p^{\max(\nu_1,\ldots,\nu_k)})}{p^{2(\nu_1 +\cdots
+\nu_k)}}.
\end{equation*}

In particular,
\begin{equation*}
C_{\sigma,2} =\zeta(3)\zeta(4) \prod_p \left(1+\frac1{p^2}
-\frac{2}{p^3}-\frac{2}{p^{5}}+\frac{2}{p^{6}}\right).
\end{equation*}
\end{corollary}

\begin{corollary} \label{Cor4} Let $k\ge 2$. Then
for every $\varepsilon >0$,
\begin{equation*}
\sum_{n_1,\ldots,n_k\le x} \varphi([n_1,\ldots,n_k]) = C_{\varphi,k}
\frac{x^{2k}}{2^k}
+O\left(x^{2k-1/2+\varepsilon}\right),
\end{equation*}
and
\begin{equation*}
\sum_{n_1,\ldots,n_k\le x} \frac{\varphi([n_1,\ldots,n_k])}{n_1\cdots
n_k} = C_{\varphi,k} x^k +O\left(x^{k-1/2+ \varepsilon}\right),
\end{equation*}
where
\begin{equation*}
C_{\varphi,k}= \prod_p \left(1-\frac1{p}\right)^k \sum_{\nu_1,\ldots,
\nu_k=0}^{\infty}
\frac{\varphi(p^{\max(\nu_1,\ldots,\nu_k)})}{p^{2(\nu_1 +\cdots
+\nu_k)}}.
\end{equation*}

In particular,
\begin{equation*}
C_{\varphi,2} =\zeta(3)\prod_p \left(1-\frac{3}{p^2}
+\frac{2}{p^3}-\frac{1}{p^4}+\frac{2}{p^5}-\frac1{p^6}\right).
\end{equation*}
\end{corollary}

\begin{corollary} \label{Cor_mu} Let $k\ge 2$. Then
for every $\varepsilon >0$,
\begin{equation*}
\sum_{n_1,\ldots,n_k\le x} \mu^2([n_1,\ldots,n_k]) = \frac{x^k}{\zeta(2)^k}
+O\left(x^{k-1/2+\varepsilon}\right).
\end{equation*}
\end{corollary}

\begin{remark} {\rm It would be interesting to find the best possible
error, especially in particular cases. For example, for $r=1$ in
Corollary \ref{Cor1}, the relative error is $O(x^{-1/2+\epsilon})$.
Can we improve the exponent further and if so, by how much?}
\end{remark}

\section{Proofs}

An arithmetic function $g$ of $k$ variables is called
multiplicative if it is not identically zero and
\begin{equation*}
g(m_1n_1,\ldots,m_kn_k) = g(m_1,\ldots,m_k) g(n_1,\ldots,n_k),
\end{equation*}
provided that $(m_1\cdots m_k,n_1\cdots n_k)=1$. Hence
\begin{equation*}
g(n_1,\ldots,n_k)= \prod_p g \left(p^{\nu_p(n_1)},
\ldots,p^{\nu_p(n_k)}\right)
\end{equation*}
for every $n_1,\ldots,n_k\in \N$. In this case
the multiple Dirichlet series of the function
$g$ can be expanded into an Euler product:
\begin{equation*}
\sum_{n_1,\ldots,n_k=1}^{\infty}
\frac{g(n_1,\ldots,n_k)}{n_1^{z_1}\cdots n_k^{z_k}} = \prod_p
\sum_{\nu_1,\ldots,\nu_k=0}^{\infty} \frac{g(p^{\nu_1},
\ldots,p^{\nu_k})}{p^{\nu_1 z_1+\cdots +\nu_k z_k}}.
\end{equation*}

We need the following lemmas.

\begin{lemma} \label{Lemma_1} If $k\ge 2$ and $f\in {\cal A}_r$ with $r>-1$ real, then
\begin{equation*}
L_{f,k}(z_1,\ldots,z_k):= \sum_{n_1,\ldots,n_k=1}^{\infty}
\frac{f([n_1,\ldots,n_k])}{n_1^{z_1}\cdots n_k^{z_k}}
= \zeta(z_1-r)\cdots \zeta(z_k-r)
H_{f,k}(z_1,\ldots,z_k),
\end{equation*}
where the multiple Dirichlet series $H_{f,k}(z_1,\ldots,z_k)$ is
absolutely convergent for
\begin{equation} \label{A}
\Re z_1,\ldots,\Re z_k > A:= \begin{cases} r+\frac1{2}, & \text{ if
$r\ge 0$}, \\  \frac{r+1}{2}, & \text{ if $-1<r<0$}.
\end{cases}
\end{equation}
\end{lemma}

\begin{proof} If $f$ is a multiplicative function of a single variable, then
the arithmetic function of $k$ variables $f([n_1,\ldots,n_k])$ is multiplicative. It follows that
\begin{equation} \label{Euler_product}
L_{f,k}(z_1,\ldots,z_k) = \prod_p \sum_{\nu_1,\ldots,\nu_k=0}^{\infty}
\frac{f(p^{\max(\nu_1,\ldots,\nu_k)})}{p^{\nu_1 z_1+\cdots +\nu_k z_k}}
\end{equation}

Case I. Assume that $r\ge 0$. Grouping the terms of the sum in \eqref{Euler_product} according to the
values $\nu_1+\cdots +\nu_k$ we have
\begin{equation} \label{product}
L_{f,k}(z_1,\ldots,z_k) = \prod_p \left(1+\frac{f(p)}{p^{z_1}}+\cdots +\frac{f(p)}{p^{z_k}}+
\sum_{\nu_1+\cdots +\nu_k\ge 2}
\frac{f(p^{\max(\nu_1,\ldots,\nu_k)})}{p^{\nu_1 z_1+\cdots +\nu_k
z_k}}\right).
\end{equation}

Let $\Re z_1,\ldots, \Re z_k\ge \delta>r$. By using condition \eqref{i} from the definition of the class ${\cal A}_r$,
\begin{equation*}
\frac{f(p)}{p^{z_j}}= \frac1{p^{z_j-r}} +O\left(\frac1{p^{\delta-r+1/2}} \right) \quad (1\le j\le k).
\end{equation*}

Also, by condition \eqref{iii} following the definition of the class ${\cal A}_r$ and by using that $r\ge 0$ we deduce that
\begin{equation*}
\left| \frac{f(p^{\max(\nu_1,\ldots,\nu_k)})}{p^{\nu_1 z_1+\cdots
+\nu_k z_k}} \right| \le C_3 \frac{p^{r \max(\nu_1,\ldots,\nu_k)}}{p^{\delta(\nu_1 +\cdots +\nu_k)}} \le C_3
\frac1{p^{(\delta-r)(\nu_1+\cdots +\nu_k)}}.
\end{equation*}

Thus the sum in \eqref{product} over $\nu_1+\cdots+\nu_k\ge 2$ is
$O\left(p^{-2(\delta-r)}\right)$. We obtain
\begin{equation*}
L_{f,k}(z_1,\ldots,z_k) \zeta^{-1}(z_1-r)\cdots \zeta^{-1}(z_k-r)
\end{equation*}
\begin{equation*}
=\prod_p \left(1-\frac1{p^{z_1-r}}\right) \cdots
\left(1-\frac1{p^{z_k-r}}\right) \left(1+\frac1{p^{z_1-r}}+\cdots
+\frac1{p^{z_k-r}}+ O \left(\frac1{p^{\delta-r+1/2}} \right) \right.
\end{equation*}
\begin{equation*}
\left. + O \left(\frac1{p^{2(\delta-r)}} \right)\right)
=\prod_p \left(1+ O \left(\frac1{p^{\delta-r+1/2}} \right)+ O \left(\frac1{p^{2(\delta-r)}} \right)\right),
\end{equation*}
since $\Re z_j \ge \delta$ ($1\le j\le k$), where the terms $\pm
\frac1{p^{z_j-r}}$ ($1\le j\le k$) cancel out. Here the latter product converges absolutely when
$\delta-r+1/2>1$ and  $2(\delta-r)>1$, that is, for $\delta>r+1/2$.

Case II. Assume that $-1<r<0$. Now we group the terms of the sum in \eqref{Euler_product} according to the
values $\max(\nu_1,\ldots,\nu_k)$:
\begin{equation} \label{product_2}
L_{f,k}(z_1,\ldots,z_k) = \prod_p \left(1+ \sum_{\max(\nu_1,\ldots,\nu_k)=1} \frac{f(p)}{p^{\nu_1 z_1+\cdots +\nu_k z_k}} +
\sum_{\max(\nu_1,\ldots,\nu_k)\ge 2} \frac{f(p^{\max(\nu_1,\ldots,\nu_k)})}{p^{\nu_1 z_1+\cdots +\nu_k z_k}}\right).
\end{equation}

Let $\Re z_1,\ldots, \Re z_k\ge \delta \ge 0$.
Consider the sum in \eqref{product_2} over $\max(\nu_1, \ldots, \nu_k)=1$ and
suppose that $\nu_i=1$ for $m$ ($1\le m\le k$) distinct values of $i$. If $m=1$, then by
condition \eqref{i} from the definition of
the class ${\cal A}_r$ we have
\begin{equation*}
\frac{f(p)}{p^{z_j}}= \frac1{p^{z_j-r}} +O\left(\frac1{p^{\delta-r+1/2}} \right) \quad (1\le j\le k).
\end{equation*}

If $m\ge 2$, then
\begin{equation*}
\left|\frac{f(p)}{p^{\nu_1 z_1+\cdots +\nu_k z_k}}\right|  \le
\frac{(C_1+1)p^r}{p^{m\delta}} = O \left(\frac1{p^{2\delta-r}}
\right).
\end{equation*}

This shows that the sum in \eqref{product_2} over $\max(\nu_1,
\ldots, \nu_k)=1$ is
\begin{equation*}
\frac1{p^{z_1-r}}+\cdots +\frac1{p^{z_k-r}}+ O \left(\frac1{p^{\delta-r+1/2}} \right) + O \left(\frac1{p^{2\delta-r}} \right).
\end{equation*}

Furthermore, by condition \eqref{ii} we deduce that for $\max(\nu_1,
\ldots, \nu_k)\ge 2$,
\begin{equation*}
\left| \frac{f(p^{\max(\nu_1,\ldots,\nu_k)})}{p^{\nu_1 z_1+\cdots
+\nu_k z_k}} \right| \le C_2 \frac{p^{r \max(\nu_1,\ldots,\nu_k)}}{p^{\delta(\nu_1 +\cdots +\nu_k)}} \le C_2
\frac1{p^{(\delta-r)\max(\nu_1,\ldots,\nu_k)}}
\end{equation*}
($\delta \ge 0$) and it follows that the sum in \eqref{product_2} over $\max(\nu_1, \ldots, \nu_k)\ge 2$ is
$O\left(p^{-2(\delta-r)}\right)= O\left(p^{-(2\delta-r)}\right)$, since $r<0$.

We obtain that
\begin{equation*}
L_{f,k}(z_1,\ldots,z_k)=  \prod_p \left(1+ \frac1{p^{z_1-r}}+\cdots +\frac1{p^{z_k-r}}+
O \left(\frac1{p^{\delta-r+1/2}} \right) + O \left(\frac1{p^{2\delta-r}} \right) \right)
\end{equation*}
and
\begin{equation*}
L_{f,k}(z_1,\ldots,z_k)\zeta^{-1}(z_1-r)\cdots \zeta^{-1}(z_k-r)
\end{equation*}
\begin{equation*}
=\prod_p \left(1-\frac1{p^{z_1-r}}\right) \cdots
\left(1-\frac1{p^{z_k-r}}\right)\prod_p \left(1+ \frac1{p^{z_1-r}}+\cdots +\frac1{p^{z_k-r}}\right.
\end{equation*}
\begin{equation*}
\left. +O \left(\frac1{p^{\delta-r+1/2}} \right) + O \left(\frac1{p^{2\delta-r}} \right) \right)
\end{equation*}
\begin{equation*}
=\prod_p \left(1+ O \left(\frac1{p^{\delta-r+1/2}} \right)+ O \left(\frac1{p^{2\delta-r}} \right)\right),
\end{equation*}
since $\Re z_j \ge \delta$ ($1\le j\le k$), where the terms $\pm
\frac1{p^{z_j-r}}$ ($1\le j\le k$) cancel out, similar to Case I.
Here the latter product converges absolutely when $\delta-r+1/2>1$ and  $2\delta-r>1$,
that is, for $\delta>(r+1)/2>0$.
\end{proof}

\begin{lemma} \label{Lemma_2} If $k\ge 2$ and $f\in {\cal A}_r$ with $r\ge 0$, then
\begin{equation*}
\overline{L}_{f,k}(z_1,\ldots,z_k):=\sum_{n_1,\ldots,n_k=1}^{\infty}
\frac{f\left(\frac{[n_1,\ldots,n_k]}{(n_1,\ldots,n_k)}\right)}{n_1^{z_1}\cdots
n_k^{z_k}}= \zeta(z_1-r)\cdots \zeta(z_k-r)
\overline{H}_{f,k}(z_1,\ldots,z_k),
\end{equation*}
where the multiple Dirichlet series $\overline{H}_{f,k}(z_1,\ldots,z_k)$ is
absolutely convergent for $\Re z_1,\ldots,\Re z_k> r+1/2$.
\end{lemma}

\begin{proof} Similar to the proof of Lemma \ref{Lemma_1}, Case I. If $f$ is multiplicative,
then the function $f([n_1,\ldots,n_k]/(n_1,\ldots, n_k))$ is also
multiplicative and we have
\begin{equation*}
\overline{L}_{f,k}(z_1,\ldots,z_k) = \prod_p
\sum_{\nu_1,\ldots,\nu_k=0}^{\infty}
\frac{f(p^{\max(\nu_1,\ldots,\nu_k)-\min(\nu_1,\ldots,\nu_k)})}{p^{\nu_1
z_1+\cdots +\nu_k z_k}}
\end{equation*}
\begin{equation} \label{product11}
= \prod_p \left(1+\frac{f(p)}{p^{z_1}}+\cdots +\frac{f(p)}{p^{z_k}}+
\sum_{\nu_1+\cdots +\nu_k\ge 2}
\frac{f(p^{\max(\nu_1,\ldots,\nu_k)-\min(\nu_1,\ldots,\nu_k)})}{p^{\nu_1
z_1+\cdots +\nu_k z_k}}\right).
\end{equation}

If $\Re z_1,\ldots, \Re z_k\ge \delta>r$, then it follows that
\begin{equation*}
\left|
\frac{f(p^{\max(\nu_1,\ldots,\nu_k)-\min(\nu_1,\ldots,\nu_k)})}{p^{\nu_1
z_1+\cdots +\nu_k z_k}} \right| \le C
\frac{p^{r(\max(\nu_1,\ldots,\nu_k)-
\min(\nu_1,\ldots,\nu_k))}}{p^{\delta(\nu_1 +\cdots +\nu_k)}} \le C
\frac1{p^{(\delta-r)(\nu_1+\cdots +\nu_k)}},
\end{equation*}
thus the sum in \eqref{product11} over $\nu_1+\cdots+\nu_k\ge 2$ is
$O\left(p^{-2(\delta-r)}\right)$. Furthermore, we use the same arguments as in
the previous proof.
\end{proof}

\begin{proof}[Proof of Theorem {\rm \ref{Th1}}]
From Lemma \ref{Lemma_1} we deduce the convolutional identity
\begin{equation*}
f([n_1,\ldots,n_k]) =\sum_{j_1 d_1=n_1,\ldots,j_k d_k=n_k}
j_1^r\cdots j_k^r h_{f,k}(d_1,\ldots,d_k),
\end{equation*}
where
\begin{equation*}
\sum_{n_1,\ldots,n_k=1}^{\infty}
\frac{h_{f,k}(n_1,\ldots,n_k)}{n_1^{z_1}\cdots n_k^{z_k}}=
H_{f,k}(z_1,\ldots,z_k).
\end{equation*}

Therefore
\begin{equation*}
\sum_{n_1,\ldots,n_k\le x} f([n_1,\ldots,n_k]) = \sum_{j_1 d_1\le
x,\ldots,j_k d_k\le x} j_1^r\cdots j_k^r h_{f,k}(d_1,\ldots,d_k)
\end{equation*}
\begin{equation*}
= \sum_{d_1,\ldots,d_k \le x} h_{f,k}(d_1,\ldots,d_k) \sum_{j_1\le
x/d_1} j_1^r \cdots \sum_{j_k\le x/d_k} j_k^r
\end{equation*}
\begin{equation*}
= \sum_{d_1,\ldots,d_k \le x} h_{f,k}(d_1,\ldots,d_k)
\left(\frac{x^{r+1}}{(r+1)d_1^{r+1}} +O(\frac{x^R}{d_1^R}) \right)
\cdots \left(\frac{x^{r+1}}{(r+1)d_k^{r+1}}+O(\frac{x^R}{d_k^R})
\right),
\end{equation*}
where $R:=\max(r,0)$. We deduce that

\begin{equation} \label{main_term}
\sum_{n_1,\ldots,n_k\le x} f([n_1,\ldots,n_k]) =
\frac{x^{k(r+1)}}{(r+1)^k} \sum_{d_1,\ldots,d_k \le x}
\frac{h_{f,k}(d_1,\ldots,d_k)}{d_1^{r+1}\cdots d_k^{r+1}}+
S_{k,r}(x),
\end{equation}
with
\begin{equation} \label{inner_sum}
S_{k,r}(x) \ll \sum_{u_1,\ldots,u_k} x^{u_1+\cdots+u_k}
\sum_{d_1,\ldots, d_k\le x}
\frac{|h_{f,k}(d_1,\ldots,d_k)|}{d_1^{u_1}\cdots d_k^{u_k}},
\end{equation}
where the first sum is over $u_1,\ldots,u_k\in \{r+1, R\}$ such that
at least one $u_i$ is $R$. Let $u_1,\ldots,u_k$ be fixed and assume
that $u_i=R$ for $t$ ($1\le t\le k$) values of $i$, we take the
first $t$ values of $i$. Then $x^{u_1+\cdots +u_k}$ times the inner sum of \eqref{inner_sum} is,
using the notation $A$ given by \eqref{A},
\begin{equation*}
\ll x^{(k-t)(r+1)+tR} \sum_{d_1,\ldots,d_k \le x}
\frac{|h_{f,k}(d_1,\ldots,d_k)|}{d_1^R \cdots d_t^R
d_{t+1}^{r+1}\cdots d_k^{r+1}}
\end{equation*}
\begin{equation*}
= x^{(k-t)(r+1)+tR} \sum_{d_1,\ldots,d_k \le x}
\frac{|h_{f,k}(d_1,\ldots,d_k)|d_1^{A-R+\varepsilon} \cdots
d_t^{A-R+\varepsilon}}{d_1^{A+\varepsilon} \cdots
d_t^{A+\varepsilon} d_{t+1}^{r+1}\cdots d_k^{r+1}}
\end{equation*}
\begin{equation*}
\le x^{(k-t)(r+1)+tR} x^{t(A-R+\varepsilon)}
\sum_{d_1,\ldots,d_k=1}^{\infty}
\frac{|h_{f,k}(d_1,\ldots,d_k)|}{d_1^{A+\varepsilon} \cdots
d_t^{A+\varepsilon} d_{t+1}^{r+1}\cdots d_k^{r+1}}
\end{equation*}
\begin{equation*}
= x^{k(r+1)-t(r+1-A)+t\varepsilon}
H_{f,k}(A+\varepsilon,\ldots,A+\varepsilon,r+1,\ldots,r+1)
\end{equation*}
\begin{equation*}
\ll x^{k(r+1)-t(r+1-A) +t\varepsilon},
\end{equation*}
since the latter series is convergent by Lemma \ref{Lemma_1}. Using that $r+1-A=\frac1{2}\min(r+1,1)>0$,
the obtained error is maximal for $t=1$ giving
\begin{equation*}
O\left(x^{k(r+1)-\frac1{2}\min(r+1,1)+\varepsilon}\right).
\end{equation*}

Furthermore, for the sum in the main term of \eqref{main_term} we
have
\begin{equation*}
\sum_{d_1,\ldots,d_k \le x}
\frac{h_{f,k}(d_1,\ldots,d_k)}{d_1^{r+1}\cdots d_k^{r+1}}
\end{equation*}
\begin{equation} \label{error2}
= \sum_{d_1,\ldots,d_k=1}^{\infty}
\frac{h_{f,k}(d_1,\ldots,d_k)}{d_1^{r+1}\cdots d_k^{r+1}} -
\sum_{\emptyset \ne I \subseteq \{1,\ldots,k\}} \sum_{\substack{d_i>x, \, i\in I\\
d_j\le x, \, j\notin I}}
\frac{h_{f,k}(d_1,\ldots,d_k)}{d_1^{r+1}\cdots d_k^{r+1}},
\end{equation}
where the series is convergent by Lemma \ref{Lemma_1}, and its sum
is $H_{f,k}(r+1,\ldots,r+1)$.

Let $I$ be fixed and assume that $I=\{1,2,\ldots,s\}$, that is
$d_1,\ldots,d_s>x$ and $d_{t+1},\ldots,d_k\le x$, where $s\ge 1$. We
deduce, by noting that $A-(r+1)= -\frac1{2} \min(r+1,1) <0$,
\begin{equation*}
\sum_{\substack{d_1,\ldots,d_s > x\\
d_{s+1},\ldots,d_k\le x}}
\frac{|h_{f,k}(d_1,\ldots,d_k)|}{d_1^{r+1}\cdots d_k^{r+1}}
\end{equation*}
\begin{equation*}
= \sum_{\substack{d_1,\ldots,d_s > x\\
d_{s+1},\ldots,d_k\le x}} \frac{|h_{f,k}(d_1,\ldots,d_k)|
d_1^{A-(r+1)+\varepsilon} \cdots
d_s^{A-(r+1)+\varepsilon}}{d_1^{A+\varepsilon}\cdots
d_s^{A+\varepsilon} d_{s+1}^{r+1}\cdots d_k^{r+1}}
\end{equation*}
\begin{equation*}
\le x^{s(A-(r+1)+\varepsilon)} \sum_{d_1,\ldots,d_k=1}^{\infty}
\frac{|h_{f,k}(d_1,\ldots,d_k)|}{d_1^{A+\varepsilon}\cdots
d_s^{A+\varepsilon} d_{s+1}^{r+1}\cdots d_k^{r+1}}
\end{equation*}
\begin{equation*}
= x^{s(A-(r+1)+\varepsilon)}
H_{f,k}(A+\varepsilon,\ldots,A+\varepsilon,r+1,\ldots,r+1)
\end{equation*}
\begin{equation*}
\ll x^{-\frac{s}{2}\min(r+1,1)+s\varepsilon},
\end{equation*}
the latter series (the same as before) being convergent, and the
obtained error is maximal for $s=1$ giving, according to
\eqref{main_term} and \eqref{error2}, the same error
\begin{equation*}
O\left(x^{k(r+1)-\frac1{2}\min(r+1,1)+\varepsilon}\right).
\end{equation*}

This proves asymptotic formula \eqref{form} with the constant
$C_{f,k}=H_{f,k}(r+1,\ldots,r+1)$. Here, according to Lemma
\ref{Lemma_1},
\begin{equation*}
C_{f,k}= \prod_p \left(1-\frac1{p}\right)^k \sum_{\nu_1,\ldots,
\nu_k=0}^{\infty}
\frac{f(p^{\max(\nu_1,\ldots,\nu_k)})}{p^{(r+1)(\nu_1 +\cdots
+\nu_k)}}.
\end{equation*}

The proof of \eqref{form_2} is similar, based on Lemma \ref{Lemma_1}
and the convolutional identity
\begin{equation*}
\frac{f([n_1,\ldots,n_k])}{(n_1\cdots n_k)^r} = \sum_{j_1
d_1=n_1,\ldots,j_k d_k=n_k}
\frac{h_{f,k}(d_1,\ldots,d_k)}{d_1^r\cdots d_k^r},
\end{equation*}
which implies that
\begin{equation*}
\sum_{n_1,\ldots,n_k\le x} \frac{f([n_1,\ldots,n_k])}{(n_1\cdots n_r)^r}
= \sum_{d_1,\ldots,d_k \le x} \frac{h_{f,k}(d_1,\ldots,d_k)}{d_1^r\cdots d_k^r} \sum_{j_1\le
x/d_1} 1 \cdots \sum_{j_k\le x/d_k} 1.
\end{equation*}
\end{proof}

\begin{proof}[Proof of Theorem {\rm \ref{Th2}}]
Formula \eqref{form_3} is obtained by using Lemma \ref{Lemma_2}, in
exactly the same way as \eqref{form} (here $r\ge 0$ and
$R=\max(r,0)=r$), with the constant
$D_{f,k}=\overline{H}_{f,k}(r+1,\ldots,r+1)$.
\end{proof}

\begin{proof}[Proof of Corollary {\rm \ref{Cor1}}]
Apply Theorem \ref{Th1} for $f=\id_r$. Here
\begin{equation*}
C_{r,3}= \prod_p \left(1-\frac1{p}\right)^3 \sum_{a,b,c=0}^{\infty}
\frac{p^{r\max(a,b,c)}}{p^{(r+1)(a+b+c)}}
\end{equation*}
\begin{equation*}
= \prod_p \left(1-\frac1{p}\right)^3 \left(6S_1+ 3S_2+3S_3+S_4
\right),
\end{equation*}
with
\begin{equation*}
S_1= \sum_{0\le a<b<c} \frac{p^{rc}}{p^{(r+1)(a+b+c)}}, \quad S_2=
\sum_{0\le a=b<c} \frac{p^{rc}}{p^{(r+1)(2a+c)}},
\end{equation*}
\begin{equation*} S_3= \sum_{0\le a<b=c}
\frac{p^{rc}}{p^{(r+1)(a+2c)}}, \quad S_4= \sum_{0\le a=b=c}
\frac{p^{rc}}{p^{(r+1)3c}},
\end{equation*}
which gives \eqref{C3r}. Formula \eqref{C4r} for the constant $C_{r,4}$ can be computed in a
similar manner.
\end{proof}

\begin{proof}[Proof of Corollary {\rm \ref{Cor2}}]
Apply Theorem \ref{Th2} for $f=\id_r$. The constants $D_{r,3}$ and
$D_{r,4}$ can be evaluated like above.
\end{proof}

\begin{proof}[Proof of Corollaries {\rm \ref{Cor3}, \ref{Cor4}, \ref{Cor_mu}}]
Apply Theorem \ref{Th1} for $f=\sigma$, $f=\varphi$ with $r=1$,
resp. $f=\mu^2$ with $r=0$.
\end{proof}

\medskip

\noindent Titus Hilberdink \\
Department of Mathematics, University of Reading \\
Whiteknights, PO Box 220, Reading RG6 6AX, UK \\
E-mail: {\tt t.w.hilberdink@reading.ac.uk}
\medskip

\noindent L\'aszl\'o T\'oth \\
Department of Mathematics, University of P\'ecs \\
Ifj\'us\'ag \'utja 6, H-7624 P\'ecs, Hungary
\\ E-mail: {\tt ltoth@gamma.ttk.pte.hu}


\begin{thebibliography}{99}
\bibitem{All1975} K.~Alladi, On generalized Euler functions and related totients,
in vol. {\it New Concepts in Arithmetic Functions}, Matscience
Report 83, The Institute of Mathematical Sciences, Madras, 1975.

\bibitem{Bor2007a} O.~Bordell\`{e}s, A note on the average order of the gcd-sum
function, {\it J. Integer Seq.} {\bf 10} (2007), Article 07.3.3, 4 pp.

\bibitem{Bor2007b} O.~Bordell\`{e}s, Mean values of generalized gcd-sum and
lcm-sum functions, {\it J. Integer Seq.} {\bf 10} (2007),
Article 07.9.2, 13 pp.

\bibitem{Bor2010} O.~Bordell\`{e}s, The composition of the gcd and certain
arithmetic functions, {\it J. Integer Seq.} {\bf 13} (2010),
Article 10.7.1, 22 pp.

\bibitem{Bro2001} K.~A.~Broughan, The gcd-sum function, {\it J. Integer Seq.} {\bf 4} (2001),
Article 01.2.2, 16 pp, errata added in 2007.

\bibitem{Ces1885} E.~Ces\`{a}ro, \`{E}tude moyenne du plus grand commun diviseur de deux
nombres, {\it Annali di Matematica Pura ed Applicata} {\bf 13}
(1885), 235--250.

\bibitem{ChiSit1985} J.~Chidambaraswamy and R.~Sitaramachandrarao, Asymptotic
results for a class of arithmetical functions, {\it Monatsh. Math.}
{\bf 99} (1985), 19--27.

\bibitem{Cil2014} J.~Cilleruelo, J.~Ru\'e, P.~ \u{S}arka,  and
A.~Zumalac\'arregui, The least common multiple of random sets of
positive integers, {\it J. Number Theory} {\bf 144} (2014), 92--104.

\bibitem{Coh1960I} E.~Cohen, Arithmetical functions of a greatest common divisor. I,
{\it Proc. Amer. Math. Soc.} {\bf 11} (1960), 164--171.

\bibitem{Coh1962II} E.~Cohen, Arithmetical functions of a greatest common divisor. II.
An alternative approach, {\it Boll. Un. Mat. Ital.} (3) {\bf 17}
(1962), 349--356.

\bibitem{Coh1961III} E.~Cohen, Arithmetical functions of a greatest common divisor. III.
Ces\`{a}ro's divisor problem, {\it Proc. Glasgow Math. Assoc.} {\bf
5} (1961), 67--75 .

\bibitem{DiaErd1977} P.~Diaconis and P.~Erd\H os, On the distribution of the
greatest common divisor, Technical Report No. 12, Department
of Statistics, Stanford University, Stanford, 1977; Reprinted in {\it A festschrift for Herman Rubin, IMS
Lecture Notes Monogr. Ser., Inst. Math. Statist.}, {\bf 45}, (2004),
56--61.

\bibitem{FF2013NT} J.~L.~Fern\'andez and P.~Fern\'andez, On the probability distribution
of the gcd and lcm of $r$-tuples of integers, Preprint, 2013, 24 pp,
arXiv:1305.0536 [math.NT].

\bibitem{Hil2014} T.~Hilberdink, The group of squarefree integers, {\it Linear Algebra Appl.}
{\bf 457} (2014), 383--399.

\bibitem{IkeMat2014} S.~Ikeda and K.~Matsuoka, On the lcm-sum function, {\it J. Integer Seq.} {\bf 17} (2014),
Article 14.1.7, 11 pp.

\bibitem{Pil1933} S.~S.~Pillai, On an arithmetic function, {\it J. Annamalai
Univ.} {\bf 2} (1933), 243--248.

\bibitem{TanZha2008} Y.~Tanigawa and W.~Zhai, On the gcd-sum function, {\it J.
Integer Seq.} {\bf 11} (2008), Article 08.2.3, 11 pp.

\bibitem{Tot2010} L.~T\'oth, A survey of gcd-sum functions, {\it J. Integer
Seq.} {\bf 13} (2010), Article 10.8.1, 23 pp.

\bibitem{Tot2013} L.~T\'oth, A survey of the alternating sum-of-divisors function,
{\it Acta Univ. Sapientiae, Math.} {\bf 5} (2013), 93--107.

\bibitem{Tot2014} L.~T\'oth, Multiplicative Arithmetic Functions of Several Variables:
A Survey, in vol. {\it Mathematics Without Boundaries, Surveys in
Pure Mathematics}, T. M. Rassias, P. M. Pardalos (eds.), Springer,
New York, 2014, pp. 483--514, arXiv:1310.7053 [math.NT].

\bibitem{Tot} L.~T\'oth, Counting $r$-tuples of positive integers with $k$-wise relatively
prime components, {\it J. Number Theory} {\bf 166} (2016), 105--116.

\bibitem{Ush2012} N.~Ushiroya, Mean-value theorems for
multiplicative arithmetic functions of several variables, {\it
Integers} {\bf 12} (2012), 989--1002.

\bibitem{Wal1963} A.~Walfisz, {\it Weylsche Exponentialsummen in der neueren Zahlentheorie},
Mathematische Forschungsberichte, XV. VEB Deutscher Verlag der Wissenschaften, Berlin, 1963.
\end{thebibliography}
\end{document}